\documentclass[11pt]{article}
\usepackage[T1]{fontenc}
\usepackage[a4paper]{anysize}\marginsize{2.0cm}{2.0cm}{1.0cm}{1.0cm}
\pdfpagewidth=\paperwidth \pdfpageheight=\paperheight
\usepackage{pgf,tikz,hyperref}
\usetikzlibrary{arrows}
\usepackage{amsfonts,amssymb,amsthm,amsmath,eucal}
\usepackage{graphicx}
\usepackage{ltablex}
\usepackage{caption, booktabs}
\usepackage{longtable}

\theoremstyle{plain}
\newtheorem{thm}{Theorem}[section]
\newtheorem{theorem}[thm]{Theorem}

\theoremstyle{definition}
\newtheorem{definition}[thm]{Definition}
\newtheorem{remark}[thm]{Remark}

\newtheorem{thevarthm}[thm]{\varthmname}

\newenvironment{varthm*}[1]{\trivlist\item[]{\bf #1.}\it}{\endtrivlist}

\newcommand\be{\begin{eqnarray*}}
\newcommand\ee{\end{eqnarray*}}

\renewcommand\P{\mathbb P}

\newcommand\newop[2]{\def#1{\mathop{\rm #2}\nolimits}}
\newop\edim{edim}
\newop\Zeroes{Zeroes}
\newop\Jac{Jac}
\newop\Ass{Ass}
\newop\SL{SL}
\newop\PGL{{\P}GL}
\newop\Km{Km}
\newop\reg{reg}

\newcolumntype{L}{>{$}l<{$}} 
\newcommand{\slfrac}[2]{\left.#1\middle/#2\right.}

\newcommand\keywords[1]{{\renewcommand\thefootnote{}\footnotetext{\textit{Keywords:} #1.}}}
\newcommand\subclass[1]{{\renewcommand\thefootnote{}\footnotetext{\textit{Mathematics Subject Classification (2010):} #1.}}}

\begin{document}

\author{Paulina Wi\'sniewska and Maciej Zi\k{e}ba}
\title{Generic projections of the $\mathrm{H}_{4}$ configuration of points}
\date{\today}
\maketitle
\thispagestyle{empty}

\begin{abstract}
In the present note we give an example of a finite set of points in $\mathbb{P}^{3}$ which has the so-called geproci property, but it is neither a grid nor a half-grid.
This answers a question on the existence of such sets rised by Pokora, Szemberg and Szpond in \cite[Problem 5.9]{MFOpreprint}.
\end{abstract}
\subclass{MSC 14C20 \and MSC 14N20 \and MSC 13A15}

\section{Introduction}
The purpose of this note is to study properties of the set of $60$ points in $\P^3$ determined by the $\mathrm{H}_4$ root system. Our motivation comes from several directions. In \cite{ChiantiniMigliore19}, the authors observed that grids in $\P^3$ have the property that their general projection to a plane is a complete intersection. Such sets are said (after Definition 5.1 in \cite{MFOpreprint}) to have the \emph{geproci} property (from: \textbf{ge}neral \textbf{pro}jection \textbf{c}omplete \textbf{i}ntersection). We recall the fundamental Definition 3.2 from \cite{ChiantiniMigliore19}.
\begin{definition}[Grid]
   Let $a$ and $b$ be positive integers. A set $Z$ of $a\cdot b$ points in $\P^3$ is an $(a,b)-$\emph{grid} if there exist two sets of lines
   $L_1,\ldots,L_a$ and $M_1,\ldots,M_b$ such that 
   \begin{itemize}
       \item lines in each of the sets are pairwise skew;
       \item each pair of lines, one from one set and one from the other intersect in a point of $Z$.
   \end{itemize} 
   Thus
   $$Z=\left\{L_i\cap M_j,\; i=1,\ldots,a,\; j=1,\ldots,b\right\}.$$
\end{definition}
   A general projection $\pi$ of an $(a,b)$--grid to a hyperplane $H\subset\P^3$ is a complete intersection. Indeed, the ideal of $\pi(Z)$ in $H$ is generated by the equations of $C=\pi(L_1\cup\ldots\cup L_a)$ and $D=\pi(M_1\cup\ldots\cup M_b)$.
   
In Appendix to \cite{ChiantiniMigliore19}, it was observed that not only grids have the geproci property. The authors discovered that the set of $24$ points in $\P^3$ determined by the $F_4$ root system does not form a grid, yet it has the geproci property. More precisely, its general projection is a complete intersection of a smooth curve of degree $4$ and a curve of degree $6$, which can be chosen to split totally in $6$ lines.

The curve of degree $4$ in contrary is uniquely determined by the projection of $F_4$ and does not split into lines. In this situation, we speak of a half-grid. This notion was introduced in \cite{MFOpreprint} by Pokora, Szpond and Szemberg. They found a set of $60$ points in $\P^3$, which, like $F_4$, is a half-grid rather than a grid.


\begin{definition}[Half-grid]
   Let $a$ and $b$ be positive integers. A set $Z$ of $a\cdot b$ points in $\P^3$ is an $(a,b)$-\emph{half-grid} if there exists a set of mutually skew lines $L_1,\ldots,L_a$ covering $Z$ and a general projection of $Z$ to
   a hyperplane is a complete intersection of images of the lines with a (possibly reducible) curve of degree $b$. 
\end{definition}
The sets with the geproci property studied in \cite{ChiantiniMigliore19} ($F_4$ and its subsets) and \cite{MFOpreprint} (Klein configuration and its subsets) are half-grids. It is natural to wonder if all geproci sets of points in $\P^3$ are half-grids. Establishing such a fact would provide a major step towards the classification of all geproci sets of points in $\P^3$. Somewhat disappointingly, we show that this is not the case. Our main result is thus the following.
\begin{varthm*}{Main Theorem}
   There exists a finite set of points in $\P^3$ which has the geproci property and which is not a half-grid.
\end{varthm*}
In the course of the proof of the Main Theorem we will see however that collinear subsets of points still play a prominent role.
\section{The $\mathrm{H}_{4}$ configuration of points}
\label{sec: H4 configuration}
In this section, we introduce the main object of our interest which is a highly symmetric set of $60$ points in $\P^3$. We discuss linear flats determined by this set and we indicate incidences among these flats that are relevant from our viewpoint.

The set of $60$ points that we are interested in comes from the $\mathrm{H}_{4}$ root system. In specific coordinates it can realized as follows: 
\begin{equation*}
\begin{tabular}{L L L}
P_{1} = \left[1:0:0:0 \right]   &
P_{2} = \left[0:1:0:0 \right]   &
P_{3} = \left[0:0:1:0 \right]\\
P_{4} = \left[0:0:0:1 \right]   &
P_{5} = \left[1:1:1:1 \right]   &
P_{6} = \left[1:1:1:-1 \right]\\

P_{7} = \left[1:1:-1:1 \right]   &
P_{8} = \left[1:1:-1:-1 \right]   &
P_{9} = \left[1:-1:1:1 \right]\\
P_{10} = \left[1:-1:1:-1 \right]  & 
P_{11} = \left[1:-1:-1:1 \right]   &
P_{12} = \left[1:-1:-1:-1 \right]\\

P_{13} = \left[0:{\varphi}:{\varphi}^2:1 \right]&
P_{14} = \left[0:{\varphi}:{\varphi}^2:-1 \right]&   
P_{15} = \left[0:{\varphi}:-{\varphi}^2:1 \right]\\

P_{16} = \left[0:{\varphi}:-{\varphi}^2:-1 \right]&   
P_{17} = \left[0:{\varphi}^2:1:{\varphi} \right]   &
P_{18} = \left[0:{\varphi}^2:1:-{\varphi} \right]\\

P_{19} = \left[0:{\varphi}^2:-1:{\varphi} \right] &  
P_{20} = \left[0:{\varphi}^2:-1:-{\varphi} \right]   &  
P_{21} = \left[0:1:{\varphi}:{\varphi}^2 \right]\\

P_{22} = \left[0:1:{\varphi}:-{\varphi}^2 \right]    & 
P_{23} = \left[0:1:-{\varphi}:{\varphi}^2 \right]    & 
P_{24} = \left[0:1:-{\varphi}:-{\varphi}^2 \right]\\

P_{25} = \left[{\varphi}:0:1:{\varphi}^2 \right]    & 
P_{26} = \left[{\varphi}:0:1:-{\varphi}^2 \right]    & 
P_{27} = \left[{\varphi}:0:-1:{\varphi}^2 \right]\\

P_{28} = \left[{\varphi}:0:-1:-{\varphi}^2 \right]    & 
P_{29} = \left[{\varphi}^2:0:{\varphi}:1 \right]    & 
P_{30} = \left[{\varphi}^2:0:{\varphi}:-1 \right]\\

P_{31} = \left[{\varphi}^2:0:-{\varphi}:1 \right]    & 
P_{32} = \left[{\varphi}^2:0:-{\varphi}:-1 \right]    & 
P_{33} = \left[1:0:{\varphi}^2:{\varphi} \right]\\

P_{34} = \left[1:0:{\varphi}^2:-{\varphi} \right]    & 
P_{35} = \left[1:0:-{\varphi}^2:{\varphi} \right]    & 
P_{36} = \left[1:0:-{\varphi}^2:-{\varphi} \right]\\

P_{37} = \left[{\varphi}:{\varphi}^2:0:1 \right]    & 
P_{38} = \left[{\varphi}:{\varphi}^2:0:-1 \right]    & 
P_{39} = \left[{\varphi}:-{\varphi}^2:0:1 \right]\\

P_{40} = \left[{\varphi}:-{\varphi}^2:0:-1 \right]    & 
P_{41} = \left[{\varphi}^2:1:0:{\varphi} \right]    & 
P_{42} = \left[{\varphi}^2:1:0:-{\varphi} \right]\\

P_{43} = \left[{\varphi}^2:-1:0:{\varphi} \right]    & 
P_{44} = \left[{\varphi}^2:-1:0:-{\varphi} \right]    & 
P_{45} = \left[1:{\varphi}:0:{\varphi}^2 \right]\\
P_{46} = \left[1:{\varphi}:0:-{\varphi}^2 \right]    & 
P_{47} = \left[1:-{\varphi}:0:{\varphi}^2 \right]   & 
P_{48} = \left[1:-{\varphi}:0:-{\varphi}^2 \right]\\

P_{49} = \left[{\varphi}:1:{\varphi}^2:0 \right]    & 
 P_{50} = \left[{\varphi}:1:-{\varphi}^2:0 \right]    & 
 P_{51} = \left[{\varphi}:-1:{\varphi}^2:0 \right]\\
 
 P_{52} = \left[{\varphi}:-1:-{\varphi}^2:0 \right]    & 
 P_{53} = \left[{\varphi}^2:{\varphi}:1:0 \right]    & 
 P_{54} = \left[{\varphi}^2:{\varphi}:-1:0 \right]\\
 
 P_{55} = \left[{\varphi}^2:-{\varphi}:1:0 \right]    & 
 P_{56} = \left[{\varphi}^2:-{\varphi}:-1:0 \right]    & 
 P_{57} = \left[1:{\varphi}^2:{\varphi}:0 \right]\\
 
 P_{58} = \left[1:{\varphi}^2:-{\varphi}:0 \right]    & 
 P_{59} = \left[1:-{\varphi}^2:{\varphi}:0 \right]    & 
 P_{60} = \left[1:-{\varphi}^2:-{\varphi}:0 \right]
\end{tabular}
\end{equation*}

\noindent
where $\varphi=\slfrac{(1+\sqrt{5})}{2}=1.61803398875\dots$ means the golden ratio.
Up to a projective change of coordinates, there is just one way to embed $\mathrm{H}_{4}$ into $\P^3$, see \cite{bao}.
This set determines $60$ dual planes $V_1,\ldots,V_{60}$, which we consider in the same $\mathbb{P}^3$. Specifically, for $P_i = \left[a:b:c:d\right]$ we define $V_i=\{ax+by+cz+dw=0\}$.
These points and planes form a symmetric $\left(60_{15}\right)$-configuration of points and planes, which means that  there are $15$ points in each plane and $15$ planes pass through each point. We call the planes $V_i$ the $15$-reach planes to indicate that they contain $15$ points from the $\mathrm{H}_4$ configuration.
Specific incidences between the configuration points and the $15$-reach planes are listed in Table \ref{tab:points_planes_incidences} in which the points are represented by numbers only.

In each of these planes, the $15$ distinguished points form a $\mathrm{H}_3$ configuration. This is indicated in Figure \ref{fig:plane}, for the plane $\mathrm{H}_4$, in affine coordinates $(x,y)$, where $\{z=0\}$ is the line at the infinity.
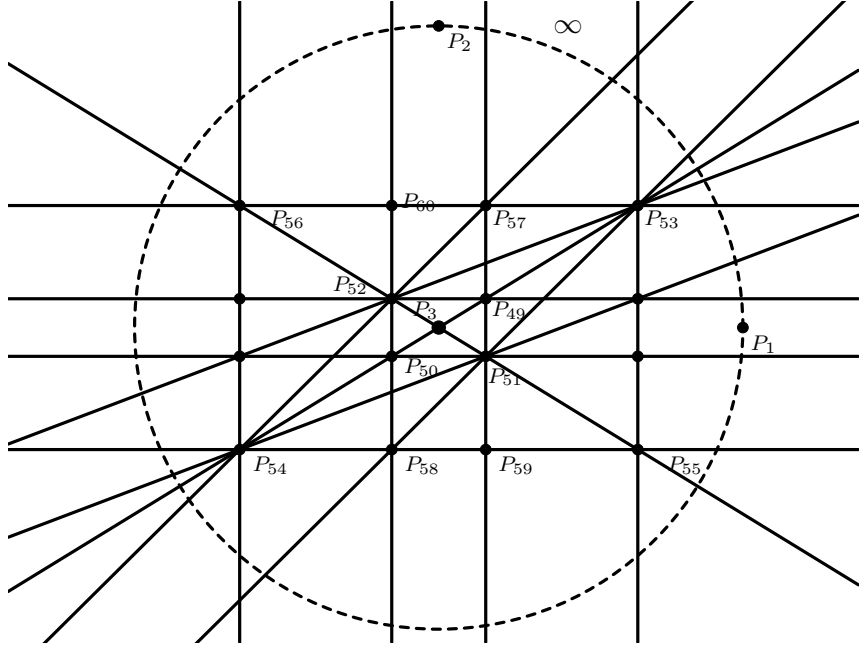
\begin{figure}[h]
    \centering

\begin{tikzpicture}[line cap=round,line join=round,>=triangle 45,x=1.0cm,y=1.0cm]
\clip(-5.664071564399751,-4.174562872034968) rectangle (5.520283925047641,4.320970581586193);
\draw [line width=1.2pt,domain=-5.664071564399751:5.520283925047641] plot(\x,{(-0.--2.*\x)/3.23606797749979});
\draw [line width=1.2pt,domain=-5.664071564399751:5.520283925047641] plot(\x,{(-0.-0.3819660112501051*\x)/0.6180339887498949});
\draw [line width=1.2pt] (0.6180339887498949,-4.174562872034968) -- (0.6180339887498949,4.320970581586193);
\draw [line width=1.2pt] (-0.6180339887498949,-4.174562872034968) -- (-0.6180339887498949,4.320970581586193);
\draw [line width=1.2pt,domain=-5.664071564399751:5.520283925047641] plot(\x,{(--3.23606797749979-0.*\x)/2.});
\draw [line width=1.2pt,domain=-5.664071564399751:5.520283925047641] plot(\x,{(-3.23606797749979-0.*\x)/2.});
\draw [line width=1.2pt,domain=-5.664071564399751:5.520283925047641] plot(\x,{(--2.--2.*\x)/2.});
\draw [line width=1.2pt,domain=-5.664071564399751:5.520283925047641] plot(\x,{(-1.2360679774997898--1.2360679774997898*\x)/1.2360679774997898});
\draw [line width=1.2pt,domain=-5.664071564399751:5.520283925047641] plot(\x,{(-0.4721359549995794-0.*\x)/1.2360679774997898});
\draw [line width=1.2pt,domain=-5.664071564399751:5.520283925047641] plot(\x,{(--0.4721359549995794-0.*\x)/1.2360679774997898});
\draw [line width=1.2pt] (2.618033988749895,-4.174562872034968) -- (2.618033988749895,4.320970581586193);
\draw [line width=1.2pt] (-2.618033988749895,-4.174562872034968) -- (-2.618033988749895,4.320970581586193);
\draw [line width=1.2pt,domain=-5.664071564399751:5.520283925047641] plot(\x,{(--1.2360679774997896--0.7639320225002102*\x)/2.});
\draw [line width=1.2pt,domain=-5.664071564399751:5.520283925047641] plot(\x,{(-2.--1.2360679774997898*\x)/3.23606797749979});
\draw [line width=1.2pt,dash pattern=on 3pt off 3pt] (0.,0.) circle (4.0cm);
\draw (1.3726329123369714,4.2) node[anchor=north west] {$\infty$};

\begin{scriptsize}
\draw [fill=black] (0.6180339887498949,0.3819660112501051) circle (2.0pt);

\draw [fill=black] (-0.6180339887498949,-0.3819660112501051) circle (2.0pt);

\draw [fill=black] (0.6180339887498949,-0.3819660112501051) circle (2.0pt);

\draw [fill=black] (-0.6180339887498949,0.3819660112501051) circle (2.0pt);

\draw [fill=black] (2.618033988749895,1.618033988749895) circle (2.0pt);

\draw [fill=black] (-2.618033988749895,-1.618033988749895) circle (2.0pt);

\draw [fill=black] (2.618033988749895,-1.618033988749895) circle (2.0pt);

\draw [fill=black] (-2.618033988749895,1.618033988749895) circle (2.0pt);

\draw [fill=black] (0.6180339887498949,1.618033988749895) circle (2.0pt);

\draw [fill=black] (-0.6180339887498949,-1.618033988749895) circle (2.0pt);

\draw [fill=black] (0.6180339887498949,-1.618033988749895) circle (2.0pt);

\draw [fill=black] (-0.6180339887498949,1.618033988749895) circle (2.0pt);

\draw [fill=black] (0.,0.) circle (2.5pt);

\draw [fill=black] (-2.618033988749895,-0.38196601125010504) circle (2.0pt);

\draw [fill=black] (-2.618033988749895,0.38196601125010504) circle (2.0pt);

\draw [fill=black] (2.6180339887498945,0.3819660112501051) circle (2.0pt);

\draw [fill=black] (2.618033988749895,-0.38196601125010504) circle (2.0pt);

\draw [fill=black] (0,4) circle (2.0pt);

\draw [fill=black] (4,0) circle (2.0pt);

\draw[color=black,anchor=north west] (0.59,0.45) node {$P_{49}$};
\draw[color=black,anchor=north west] (-0.56,-0.31) node {$P_{50}$};
\draw[color=black,anchor=north west] (0.54,-0.43) node {$P_{51}$};
\draw[color=black,anchor=north west] (-1.5,0.78) node {$P_{52}$};
\draw[color=black,anchor=north west] (-0.45,0.45) node {$P_{3}$};
\draw[color=black,anchor=north west] (2.6,1.66) node {$P_{53}$};
\draw[color=black,anchor=north west] (-2.55,-1.6) node {$P_{54}$};
\draw[color=black,anchor=north west] (2.9,-1.6) node {$P_{55}$};
\draw[color=black,anchor=north west] (-2.33,1.66) node {$P_{56}$};
\draw[color=black,anchor=north west] (0.6,1.66) node {$P_{57}$};
\draw[color=black,anchor=north west] (-0.5528391117378729,-1.6) node {$P_{58}$};
\draw[color=black,anchor=north west] (0.68,-1.6) node {$P_{59}$};
\draw[color=black,anchor=west] (-0.6,1.66) node {$P_{60}$};

\draw[color=black,anchor=north west] (0,4) node {$P_{2}$};
\draw[color=black,anchor=north west] (4,0) node {$P_{1}$};
\end{scriptsize}
\end{tikzpicture}

    \caption{The $\mathrm{H}_{3}$ configuration of points on one of $60$ distinguished planes (here the $V_4$ plane)}
    \label{fig:plane}
\end{figure}
Additional collinearities, which are of interest to us, determine $72$ lines which contain exactly $5$ configuration points (and which we call $5$-reach lines for this reason). Moreover, $5$ is the maximal number of collinear configuration points.

Through each point of configuration there are exactly six $5$-reach lines passing. The incidences are indicated in Table \ref{tab:points_lines_incidences}, where points are represented again by numbers only. 

Because of the duality between points $P_i$ and planes $V_i$, each of the $5$-reach lines is contained in exactly $5$ of the $15$-reach planes and each of these planes contain exactly six $5$-reach lines. 


\section{The proof of the Main Theorem.}
In this section, we prove the Main Theorem. We begin by showing that the set of $60$ points determined by the $\mathrm{H}_{4}$ root system has the geproci property.
\begin{theorem}\label{thm: 60 is geproci}
   Let $Z$ be the set of $60$ points defined in Section \ref{sec: H4 configuration}. Then a general projection of $Z$ to $\P^2$ is a $(6,10)$ complete intersection.
\end{theorem}
\begin{proof}
   Let $P$ be a general point in $\P^3$. Thus, in particular, $P$ is not contained in any of flats described in Section \ref{sec: H4 configuration}. Let
   $$\pi_P:\P^3\dashrightarrow\P^2$$
   be the projection from $P$.
   
   The existence of a unique cone of degree $6$ with vertex at $P$ vanishing along the set $Z$ was established in \cite[Section 3.7]{HMNT}. We checked, using our Singular script, that this cone is smooth apart from its vertex. It implies that $\pi_P(Z)$ is contained in a unique smooth curve $C_6$ of degree $6$. 
   
   Thus, it remains to establish the existence of a curve of degree $10$, not containing $C_6$ as a component, and vanishing at all points of $\pi_P(Z)$. We reveal additional nice properties of the set $Z$.
   
   To begin with please note that the following two sets of lines:
   $$L_1=\ell_1, L_2=\ell_{25}, L_3=\ell_{32}, L_4=\ell_{37}, L_5=\ell_{44},\;\;
   M_1=\ell_2, M_2=\ell_{26}, M_3=\ell_{31}, M_4=\ell_{38}, M_5=\ell_{43}$$
   form a $(5,5)$-grid which consists of $25$ points with the following numbers:
   $$1,5,6,7,8,13,14,15,16,29,30,\ldots,38,41,42,51,52,57,58.$$
   It was established in \cite[Remark 3.4]{ChiantiniMigliore19} that all $(a,b)$-grids with $a,b\geq 3$ are contained in a quadric. In our case the equation of this quadric is the following:
   $$Q_1:\; 2xy-y^2+(\varphi-1)z^2-\varphi w^2=0.$$
   Let $g_i$ be the equation of the plane generated by $L_i$ and $P$ with $i \in \{1, ..., 5\}$. Similarly, denote by $h_i$ the equation of the plane spanned by $M_i$ and $P$ with $i \in \{1, ..., 5\}$.
   Then the products
   $$g=g_1\cdot\ldots\cdot g_5\;\mbox{ and }\; 
   h=h_1\cdot\ldots\cdot h_5$$
   generate a pencil of cones of degree $5$ with vertex at $P$ vanishing at all points of the $(5,5)$-grid. It is easy to check that this pencil has no additional base lines apart of those joining $P$ and the points in the grid. We impose one more vanishing condition by requiring a member of this pencil to vanish at $P_4$. This point is distinguished by the following property: There are $10$ lines passing through $P_4$ which meet the quadric $Q_1$ in two points from the grid. More precisely the points in the grid are grouped by these incidences in pairs as follows:
   $$\{\{5,6\},\{7,8\},\{13,14\},\{15,16\},\{29,30\},\{31,32\},\{33,34\},\{35,36\},\{37,38\},\{41,42\}\}.$$
   There are exactly $4$ more points in the configuration which have the same property, i.e., for each of them there are $10$ lines intersecting $Q_1$ in pairs of points from the grid. These points are: $P_{39}, P_{40}, P_{47}$ and $P_{48}$. Together with $P_4$ these points lie on the line $\ell_{24}$. It is easy to check by computer that the member of the pencil vanishing at $P_4$ vanishes at all other points of the configuration contained in the line $\ell_{24}$. 
   
   Thus in this part of the proof we established the existence of a cone $C_5$ of degree $5$ with vertex at $P$ vanishing at half of the points of the configuration:
   $$Z_1:\; 1,4,5,6,7,8,13,14,15,16,29,30,\ldots,42,47,48,51,52,57,58.$$
   This cone is smooth away of its vertex, hence there exists a smooth curve $\Gamma_5$ vanishing at all points of $\pi_P(Z_1)$.
   These $30$ points are by construction contained in $6$ disjoint lines. One can take, for example, lines from the $L$ set and line $\ell_{24}$, so that we have
   $$Z_1\subset \ell_1\cup\ell_{24}\cup\ell_{25}\cup
   \ell_{32}\cup \ell_{37}\cup \ell_{44}.$$
   This implies that $\pi(Z_1)$ is a $(5,6)$-complete intersection. In particular, $Z_1$ has the geproci property and it is a half-grid!
   
   It is natural to wonder properties does the residual set
   $$Z_2=Z\setminus Z_1$$
   enjoy. In a sense it is surprising that exactly the same as $Z_1$.
   We have a grid determined by lines
   $$L_1'=\ell_{7}, L_2'=\ell_{51}, L_3'=\ell_{60}, L_4'=\ell_{65}, L_5'=\ell_{70},\;\;
   M_1'=\ell_{8}, M_2'=\ell_{54}, M_3'=\ell_{58}, M_4'=\ell_{63}, M_5'=\ell_{71},$$
   which is contained in the quadric
   $$Q_2: x^2+2xy+\varphi z^2-\left(\varphi - 1\right) w^2 .$$
   The line external to the quadric is $\ell_{17}$.
   As before, the projection of $Z_2$ to $\P^2$ is the intersection of a smooth curve $\Gamma_5'$ of degree $5$ and projection of $6$ lines, for example the lines $ \ell_{7}, \ell_{17}, \ell_{51}, \ell_{60}, \ell_{65}, \ell_{70}$. So $Z_2$ is also a geproci set and a half-grid.
   
   The upshot of these considerations is that $Z$ is contained in a cone of degree $10$ with vertex at $P$, namely the union $C_5 \cup C_5'$ and consequently $\pi_P(Z)$ is a $(6,10)$-complete intersection.
\end{proof}
Now the rest of the proof of Main Theorem follows easily. It remains only to check that $Z$ is not a half-grid. The reason is that there are no lines containing $6$ or more points from $Z$.
\begin{remark}
It is worth pointing out that even if $Z$ is not a half-grid, it is a union of two such sets. In particular, it can be completely covered by a union of $12$ skew $5$-reach lines. \\
Interestingly, there are $84$ different ways to cover $Z$ by $12$ disjoint $5$-reach lines. These covering are indicated in Table \ref{fig:lines_covering}, where this time lines are represented by numbers only.
\end{remark}


\paragraph*{Acknowledgement.}

We would like to thank Piotr Pokora, Tomasz Szemberg and Justyna Szpond for their guidance and valuable comments. This research has been initiated during a workshop held in Lanckorona in Autumn 2020 and continued in Lanckorona in late Spring 2021. 

The first author was partially supported by the National Science Center (Poland) Sonata Grant Nr \textbf{2018/31/D/ST1/00177}.

The second author is partially supported by the Polish Ministry of Science and Higher Education within the program “Najlepsi z najlepszych 4.0” (”The best of the best 4.0”).

Finally, we would like to thank the referee for all of the insightful remarks and recommendations that helped us enhance the article. 




\bigskip
\bigskip
\small

\bigskip
\noindent
   Paulina Wi\'sniewska,\\
   Department of Mathematics, Pedagogical University of Cracow,
   Podchor\k a\.zych 2,
   PL-30-084 Krak\'ow, Poland.

\nopagebreak
\noindent
   \textit{E-mail address:} \texttt{paulina.wisniewska@doktorant.up.krakow.pl}\\

\bigskip
\noindent
   Maciej Zi\c{e}ba,\\
Department of Mathematics, Pedagogical University of Cracow,
   Podchor\k a\.zych 2,
   PL-30-084 Krak\'ow, Poland.
   
\nopagebreak
\noindent
   \textit{E-mail address:} \texttt{maciej.zieba@doktorant.up.krakow.pl}\\


\newpage
\section*{Appendix}
Here we present the incidences between points in the configuration and their dual planes.

    
\footnotesize
\begin{longtable}{ l  l }
$V_{1}: 2, 3, 4, 13, 14, 15, 16, 17, 18, 19, 20, 21, 22, 23, 24 $ &
$V_{2}: 1, 3, 4, 25, 26, 27, 28, 29, 30, 31, 32, 33, 34, 35, 36 $ \\
$V_{3}: 1, 2, 4, 37, 38, 39, 40, 41, 42, 43, 44, 45, 46, 47, 48 $ &
$V_{4}: 1, 2, 3, 49, 50, 51, 52, 53, 54, 55, 56, 57, 58, 59, 60 $ \\
$V_{5}: 8, 10, 11, 15, 20, 22, 26, 32, 35, 39, 44, 46, 50, 56, 59 $ &
$V_{6}: 7, 9, 12, 16, 19, 21, 25, 31, 36, 40, 43, 45, 50, 56, 59 $ \\
$V_{7}: 6, 9, 12, 13, 18, 24, 28, 30, 33, 39, 44, 46, 49, 55, 60 $ &
$V_{8}: 5, 10, 11, 14, 17, 23, 27, 29, 34, 40, 43, 45, 49, 55, 60 $ \\
$V_{9}: 6, 7, 12, 14, 17, 23, 26, 32, 35, 37, 42, 48, 52, 54, 57 $ &
$V_{10}: 5, 8, 11, 13, 18, 24, 25, 31, 36, 38, 41, 47, 52, 54, 57 $ \\
$V_{11}: 5, 8, 10, 16, 19, 21, 28, 30, 33, 37, 42, 48, 51, 53, 58 $ &
$V_{12}: 6, 7, 9, 15, 20, 22, 27, 29, 34, 38, 41, 47, 51, 53, 58 $ \\
$V_{13}: 1, 7, 10, 20, 23, 26, 27, 42, 43, 46, 47, 54, 55, 58, 59 $ &
$V_{14}: 1, 8, 9, 19, 24, 25, 28, 41, 44, 45, 48, 54, 55, 58, 59 $ \\
$V_{15}: 1, 5, 12, 18, 21, 25, 28, 42, 43, 46, 47, 53, 56, 57, 60 $ &
$V_{16}: 1, 6, 11, 17, 22, 26, 27, 41, 44, 45, 48, 53, 56, 57, 60 $ \\
$V_{17}: 1, 8, 9, 16, 22, 30, 31, 34, 35, 42, 43, 46, 47, 50, 51 $ &
$V_{18}: 1, 7, 10, 15, 21, 29, 32, 33, 36, 41, 44, 45, 48, 50, 51 $ \\
$V_{19}: 1, 6, 11, 14, 24, 29, 32, 33, 36, 42, 43, 46, 47, 49, 52 $ &
$V_{20}: 1, 5, 12, 13, 23, 30, 31, 34, 35, 41, 44, 45, 48, 49, 52 $ \\
$V_{21}: 1, 6, 11, 15, 18, 30, 31, 34, 35, 38, 39, 54, 55, 58, 59 $ &
$V_{22}: 1, 5, 12, 16, 17, 29, 32, 33, 36, 37, 40, 54, 55, 58, 59 $ \\
$V_{23}: 1, 8, 9, 13, 20, 29, 32, 33, 36, 38, 39, 53, 56, 57, 60 $ &
$V_{24}: 1, 7, 10, 14, 19, 30, 31, 34, 35, 37, 40, 53, 56, 57, 60 $ \\
$V_{25}: 2, 6, 10, 14, 15, 32, 34, 38, 40, 42, 44, 50, 52, 58, 60 $ &
$V_{26}: 2, 5, 9, 13, 16, 31, 33, 37, 39, 41, 43, 50, 52, 58, 60 $ \\
$V_{27}: 2, 8, 12, 13, 16, 30, 36, 38, 40, 42, 44, 49, 51, 57, 59 $ &
$V_{28}: 2, 7, 11, 14, 15, 29, 35, 37, 39, 41, 43, 49, 51, 57, 59 $ \\
$V_{29}: 2, 8, 12, 18, 19, 22, 23, 28, 35, 46, 48, 50, 52, 58, 60 $ &
$V_{30}: 2, 7, 11, 17, 20, 21, 24, 27, 36, 45, 47, 50, 52, 58, 60 $ \\
$V_{31}: 2, 6, 10, 17, 20, 21, 24, 26, 33, 46, 48, 49, 51, 57, 59 $ &
$V_{32}: 2, 5, 9, 18, 19, 22, 23, 25, 34, 45, 47, 49, 51, 57, 59 $ \\
$V_{33}: 2, 7, 11, 18, 19, 22, 23, 26, 31, 38, 40, 42, 44, 54, 56 $ &
$V_{34}: 2, 8, 12, 17, 20, 21, 24, 25, 32, 37, 39, 41, 43, 54, 56 $ \\
$V_{35}: 2, 5, 9, 17, 20, 21, 24, 28, 29, 38, 40, 42, 44, 53, 55 $ &
$V_{36}: 2, 6, 10, 18, 19, 22, 23, 27, 30, 37, 39, 41, 43, 53, 55 $ \\
$V_{37}: 3, 9, 11, 22, 24, 26, 28, 34, 36, 44, 47, 51, 52, 55, 56 $ &
$V_{38}: 3, 10, 12, 21, 23, 25, 27, 33, 35, 43, 48, 51, 52, 55, 56 $ \\
$V_{39}: 3, 5, 7, 21, 23, 26, 28, 34, 36, 42, 45, 49, 50, 53, 54 $ &
$V_{40}: 3, 6, 8, 22, 24, 25, 27, 33, 35, 41, 46, 49, 50, 53, 54 $ \\
$V_{41}: 3, 10, 12, 14, 16, 18, 20, 26, 28, 34, 36, 40, 46, 59, 60 $ &
$V_{42}: 3, 9, 11, 13, 15, 17, 19, 25, 27, 33, 35, 39, 45, 59, 60 $ \\
$V_{43}: 3, 6, 8, 13, 15, 17, 19, 26, 28, 34, 36, 38, 48, 57, 58 $ &
$V_{44}: 3, 5, 7, 14, 16, 18, 20, 25, 27, 33, 35, 37, 47, 57, 58 $ \\
$V_{45}: 3, 6, 8, 14, 16, 18, 20, 30, 32, 39, 42, 51, 52, 55, 56 $ &
$V_{46}: 3, 5, 7, 13, 15, 17, 19, 29, 31, 40, 41, 51, 52, 55, 56 $ \\
$V_{47}: 3, 10, 12, 13, 15, 17, 19, 30, 32, 37, 44, 49, 50, 53, 54 $ &
$V_{48}: 3, 9, 11, 14, 16, 18, 20, 29, 31, 38, 43, 49, 50, 53, 54 $ \\
$V_{49}: 4, 7, 8, 19, 20, 27, 28, 31, 32, 39, 40, 47, 48, 56, 58 $ &
$V_{50}: 4, 5, 6, 17, 18, 25, 26, 29, 30, 39, 40, 47, 48, 55, 57 $ \\
$V_{51}: 4, 11, 12, 17, 18, 27, 28, 31, 32, 37, 38, 45, 46, 54, 60 $ &
$V_{52}: 4, 9, 10, 19, 20, 25, 26, 29, 30, 37, 38, 45, 46, 53, 59 $ \\
$V_{53}: 4, 11, 12, 15, 16, 23, 24, 35, 36, 39, 40, 47, 48, 52, 59 $ &
$V_{54}: 4, 9, 10, 13, 14, 21, 22, 33, 34, 39, 40, 47, 48, 51, 60 $ \\
$V_{55}: 4, 7, 8, 13, 14, 21, 22, 35, 36, 37, 38, 45, 46, 50, 57 $ &
$V_{56}: 4, 5, 6, 15, 16, 23, 24, 33, 34, 37, 38, 45, 46, 49, 58 $ \\
$V_{57}: 4, 9, 10, 15, 16, 23, 24, 27, 28, 31, 32, 43, 44, 50, 55 $ &
$V_{58}: 4, 11, 12, 13, 14, 21, 22, 25, 26, 29, 30, 43, 44, 49, 56 $ \\
$V_{59}: 4, 5, 6, 13, 14, 21, 22, 27, 28, 31, 32, 41, 42, 52, 53 $ &
$V_{60}: 4, 7, 8, 15, 16, 23, 24, 25, 26, 29, 30, 41, 42, 51, 54 $ \\
\caption{Incidences of points and $15$-reach planes}
\label{tab:points_planes_incidences}
\end{longtable}\normalsize


The next table indicates incidences between the configuration points and the $5$-reach lines.
\footnotesize

\begin{longtable}{l l l }
$\ell_{1}: 1,29,32,33,36 $ &
$\ell_{2}: 1,30,31,34,35 $ &
$\ell_{3}: 1,41,44,45,48 $ \\
$\ell_{4}: 1,42,43,46,47 $ &
$\ell_{5}: 1,53,56,57,60 $ &
$\ell_{6}: 1,54,55,58,59 $ \\
$\ell_{7}: 2,17,20,21,24 $ &
$\ell_{8}: 2,18,19,22,23 $ &
$\ell_{9}: 2,37,39,41,43 $ \\
$\ell_{10}: 2,38,40,42,44 $ &
$\ell_{11}: 2,49,51,57,59 $ &
$\ell_{12}: 2,50,52,58,60 $ \\
$\ell_{13}: 3,13,15,17,19 $ &
$\ell_{14}: 3,14,16,18,20 $ &
$\ell_{15}: 3,25,27,33,35 $ \\
$\ell_{16}: 3,26,28,34,36 $ &
$\ell_{17}: 3,49,50,53,54 $ &
$\ell_{18}: 3,51,52,55,56 $ \\
$\ell_{19}: 4,13,14,21,22 $ &
$\ell_{20}: 4,15,16,23,24 $ &
$\ell_{21}: 4,25,26,29,30 $ \\
$\ell_{22}: 4,27,28,31,32 $ &
$\ell_{23}: 4,37,38,45,46 $ &
$\ell_{24}: 4,39,40,47,48 $ \\
$\ell_{25}: 5,13,31,41,52 $ &
$\ell_{26}: 5,16,33,37,58 $ &
$\ell_{27}: 5,17,29,40,55 $ \\
$\ell_{28}: 5,18,25,47,57 $ &
$\ell_{29}: 5,21,28,42,53 $ &
$\ell_{30}: 5,23,34,45,49 $ \\
$\ell_{31}: 6,14,32,42,52 $ &
$\ell_{32}: 6,15,34,38,58 $ &
$\ell_{33}: 6,17,26,48,57 $ \\
$\ell_{34}: 6,18,30,39,55 $ &
$\ell_{35}: 6,22,27,41,53 $ &
$\ell_{36}: 6,24,33,46,49 $ \\
$\ell_{37}: 7,14,35,37,57 $ &
$\ell_{38}: 7,15,29,41,51 $ &
$\ell_{39}: 7,19,31,40,56 $ \\
$\ell_{40}: 7,20,27,47,58 $ &
$\ell_{41}: 7,21,36,45,50 $ &
$\ell_{42}: 7,23,26,42,54 $ \\
$\ell_{43}: 8,13,36,38,57 $ &
$\ell_{44}: 8,16,30,42,51 $ &
$\ell_{45}: 8,19,28,48,58 $ \\
$\ell_{46}: 8,20,32,39,56 $ &
$\ell_{47}: 8,22,35,46,50 $ &
$\ell_{48}: 8,24,25,41,54 $ \\
$\ell_{49}: 9,13,33,39,60 $ &
$\ell_{50}: 9,16,31,43,50 $ &
$\ell_{51}: 9,19,25,45,59 $ \\
$\ell_{52}: 9,20,29,38,53 $ &
$\ell_{53}: 9,22,34,47,51 $ &
$\ell_{54}: 9,24,28,44,55 $ \\
$\ell_{55}: 10,14,34,40,60 $ &
$\ell_{56}: 10,15,32,44,50 $ &
$\ell_{57}: 10,19,30,37,53 $ \\
$\ell_{58}: 10,20,26,46,59 $ &
$\ell_{59}: 10,21,33,48,51 $ &
$\ell_{60}: 10,23,27,43,55 $ \\
$\ell_{61}: 11,14,29,43,49 $ &
$\ell_{62}: 11,15,35,39,59 $ &
$\ell_{63}: 11,17,27,45,60 $ \\
$\ell_{64}: 11,18,31,38,54 $ &
$\ell_{65}: 11,22,26,44,56 $ &
$\ell_{66}: 11,24,36,47,52 $ \\
$\ell_{67}: 12,13,30,44,49 $ &
$\ell_{68}: 12,16,36,40,59 $ &
$\ell_{69}: 12,17,32,37,54 $ \\
$\ell_{70}: 12,18,28,46,60 $ &
$\ell_{71}: 12,21,25,43,56 $ &
$\ell_{72}: 12,23,35,48,52 $ \\
\caption{Incidences of points and $5$-reach line}
\label{tab:points_lines_incidences}
\end{longtable}

    

\begin{table}[]
    \centering
    \setlength{\tabcolsep}{10pt}
\renewcommand{\arraystretch}{1.1}
\begin{tabular}{|l||l|}\hline
1, 7, 17, 24, 25, 32, 37, 44, 51, 60, 65, 70 &
1, 7, 18, 23, 28, 35, 42, 45, 50, 55, 62, 67\\ \hline 
1, 8, 17, 23, 25, 33, 40, 44, 54, 55, 62, 71 & 
1, 8, 17, 24, 25, 32, 37, 44, 54, 58, 63, 71\\ \hline 
1, 8, 18, 23, 29, 33, 40, 48, 50, 55, 62, 67 & 
1, 8, 18, 24, 29, 32, 37, 48, 50, 58, 63, 67\\ \hline
1, 9, 17, 19, 28, 32, 39, 44, 54, 58, 63, 72  &
1, 9, 18, 20, 29, 33, 40, 47, 51, 55, 64, 67 \\ \hline
1, 10, 17, 20, 25, 34, 37, 45, 53, 58, 63, 71 &
1, 10, 18, 19, 30, 33, 40, 48, 50, 57, 62, 70\\ \hline
1, 11, 13, 23, 29, 34, 40, 48, 50, 55, 65, 72 &
1, 11, 14, 24, 25, 32, 42, 47, 54, 57, 63, 71\\ \hline
1, 12, 13, 24, 30, 35, 37, 44, 54, 58, 64, 71 &
1, 12, 14, 23, 29, 33, 39, 48, 53, 60, 62, 67\\ \hline
2, 7, 17, 23, 28, 31, 38, 45, 49, 60, 65, 68 &
2, 7, 17, 24, 26, 31, 38, 43, 51, 60, 65, 70\\ \hline
2, 7, 18, 23, 28, 35, 42, 45, 49, 56, 61, 68 &
2, 7, 18, 24, 26, 35, 42, 43, 51, 56, 61, 70\\ \hline
2, 8, 17, 24, 26, 31, 38, 43, 54, 58, 63, 71 &
2, 8, 18, 23, 29, 33, 40, 48, 49, 56, 61, 68\\ \hline
2, 9, 17, 20, 27, 31, 40, 43, 51, 59, 65, 70 &
2, 9, 18, 19, 28, 36, 42, 45, 52, 56, 63, 68\\ \hline
2, 10, 17, 19, 26, 33, 38, 46, 51, 60, 66, 70 &
2, 10, 18, 20, 28, 35, 41, 45, 49, 58, 61, 69\\ \hline
2, 11, 13, 24, 26, 31, 41, 48, 52, 60, 65, 70 &
2, 11, 14, 23, 27, 35, 42, 45, 49, 56, 66, 71\\ \hline
2, 12, 13, 23, 28, 35, 42, 46, 54, 59, 61, 68 &
2, 12, 14, 24, 29, 36, 38, 43, 51, 60, 65, 69\\ \hline
3, 7, 17, 21, 26, 31, 39, 43, 53, 60, 62, 70 &
3, 7, 18, 22, 28, 32, 42, 47, 49, 57, 61, 68\\ \hline
3, 8, 17, 22, 27, 32, 37, 44, 49, 58, 66, 71 &
3, 8, 18, 21, 29, 36, 40, 43, 50, 55, 62, 69\\ \hline
3, 11, 13, 22, 26, 34, 42, 47, 52, 55, 66, 71 &
3, 11, 14, 21, 29, 32, 39, 47, 49, 60, 66, 69\\ \hline
3, 11, 14, 22, 27, 32, 42, 47, 49, 57, 66, 71 &
3, 11, 15, 19, 27, 32, 42, 46, 50, 57, 66, 70\\ \hline
3, 11, 16, 20, 27, 31, 40, 47, 49, 57, 64, 71 &
3, 12, 13, 21, 29, 36, 37, 46, 53, 60, 64, 68\\ \hline
3, 12, 14, 21, 29, 36, 39, 43, 53, 60, 62, 69 &
3, 12, 14, 22, 27, 36, 42, 43, 53, 57, 62, 71\\ \hline
3, 12, 15, 20, 29, 34, 39, 43, 53, 58, 61, 69 &
3, 12, 16, 19, 28, 36, 39, 44, 52, 60, 62, 69\\ \hline
4, 7, 17, 22, 26, 34, 38, 43, 51, 55, 65, 72 &
4, 7, 18, 21, 30, 35, 37, 45, 49, 56, 64, 68\\ \hline
4, 8, 17, 21, 25, 32, 37, 46, 54, 59, 63, 68 &
4, 8, 18, 22, 26, 33, 41, 48, 52, 55, 62, 67\\ \hline
4, 11, 13, 21, 26, 35, 41, 46, 54, 55, 64, 72 &
4, 11, 13, 22, 26, 34, 41, 48, 52, 55, 65, 72\\ \hline
4, 11, 14, 22, 27, 32, 41, 48, 49, 57, 65, 72 &
4, 11, 15, 20, 25, 34, 41, 45, 52, 55, 65, 69\\ \hline
4, 11, 16, 19, 26, 34, 39, 48, 52, 56, 63, 72 &
4, 12, 13, 21, 30, 35, 37, 46, 54, 59, 64, 68\\ \hline
4, 12, 13, 22, 30, 34, 37, 48, 52, 59, 65, 68 &
4, 12, 14, 21, 30, 35, 39, 43, 54, 59, 62, 69\\ \hline
4, 12, 15, 19, 30, 33, 38, 46, 54, 57, 64, 68 &
4, 12, 16, 20, 27, 35, 37, 46, 51, 59, 64, 67\\ \hline
5, 7, 15, 23, 25, 34, 42, 45, 53, 56, 61, 68 &
5, 7, 16, 24, 26, 31, 38, 47, 51, 60, 64, 67\\ \hline
5, 8, 15, 24, 25, 32, 41, 44, 54, 58, 61, 69 &
5, 8, 16, 23, 27, 31, 40, 48, 50, 59, 62, 67\\ \hline
5, 9, 13, 21, 30, 31, 40, 47, 54, 59, 64, 68 &
5, 9, 14, 22, 27, 32, 42, 47, 51, 59, 66, 67\\ \hline
5, 9, 15, 20, 27, 31, 41, 45, 53, 58, 64, 67 &
5, 9, 16, 19, 27, 36, 40, 44, 51, 56, 64, 72\\ \hline
5, 9, 16, 20, 27, 31, 40, 47, 51, 59, 64, 67 &
5, 10, 13, 22, 26, 34, 41, 48, 53, 58, 61, 72\\ \hline
5, 10, 14, 21, 25, 36, 41, 45, 53, 60, 62, 69 &
5, 10, 15, 19, 30, 34, 38, 45, 50, 58, 66, 69\\ \hline
5, 10, 15, 20, 25, 34, 41, 45, 53, 58, 61, 69 &
5, 10, 16, 20, 25, 34, 40, 47, 51, 59, 61, 69\\ \hline
6, 7, 15, 24, 30, 31, 38, 43, 50, 57, 65, 70 &
6, 7, 16, 23, 28, 35, 39, 44, 49, 56, 61, 72\\ \hline
6, 8, 15, 23, 29, 33, 38, 46, 50, 55, 66, 67 &
6, 8, 16, 24, 25, 36, 37, 44, 52, 56, 63, 71\\ \hline
6, 9, 13, 22, 28, 36, 41, 44, 52, 55, 65, 72 &
6, 9, 14, 21, 29, 36, 39, 43, 53, 56, 63, 72\\ \hline
6, 9, 15, 19, 30, 33, 39, 44, 52, 56, 66, 70 &
6, 9, 16, 19, 28, 36, 39, 44, 52, 56, 63, 72\\ \hline
6, 9, 16, 20, 28, 31, 39, 47, 52, 59, 63, 67 &
6, 10, 13, 21, 30, 35, 37, 46, 50, 59, 66, 70\\ \hline
6, 10, 14, 22, 30, 33, 38, 47, 49, 57, 66, 71 &
6, 10, 15, 19, 30, 33, 38, 46, 50, 57, 66, 70\\ \hline
6, 10, 15, 20, 25, 33, 41, 46, 53, 57, 61, 70 &
6, 10, 16, 19, 28, 36, 38, 46, 50, 57, 63, 72\\ \hline
\end{tabular}

    \caption{Sets of disjoint lines whose unions contain every point of the $\mathrm{H}_{4}$ configuration}
    \label{fig:lines_covering}
\end{table}

\end{document}